\documentclass[11pt]{article}
\usepackage{mathtools}
\usepackage{amsmath}
\usepackage{amssymb}
\usepackage{amsthm}
\usepackage{multirow}
\usepackage{color}
\usepackage{longtable}
\usepackage{array}
\usepackage{url}
\usepackage{caption}
\usepackage{tikz}

\usepackage{floatrow}
\DeclareFloatFont{small}{\small}

\let\emptyset\varnothing
\newtheorem{theorem}{Theorem}[section]

\newtheorem{lemma}[theorem]{Lemma}
\newtheorem{example}[theorem]{Example}

\newtheorem{corollary}[theorem]{Corollary}

\newtheorem{remark}{Remark}[section]

\title{New Partial Geometric Difference Sets and Partial Geometric Difference Families}

\author{
Jerod Michel\thanks{Corresponding author. Email Address: contextolibre@gmail.com.}
\thanks{J. Michel is with the School of Mathematics, Zhejiang University, Hangzhou 310027, China.} \\
}

\begin{document}

\date{}\maketitle

\begin{abstract}

Olmez, in ``Symmetric $1\frac{1}{2}$-Designs and $1\frac{1}{2}$-Difference Sets'' (2014), introduced the concept of a partial geometric difference set (also referred to as a $1\frac{1}{2}$-design), and showed that partial geometric difference sets give partial geometric designs. Nowak et al., in ``Partial Geometric Difference Families'' (2014), introduced the concept of a partial difference family, and showed that these also give partial geometric designs. It was shown by Brouwer et al. in ``Directed strongly regular graphs from $1\frac{1}{2}$-designs'' (2012) that directed strongly regular graphs can be obtained from partial geometric designs. In this correspondence we construct several families of partial geometric difference sets and partial difference families with new parameters, thereby giving directed strongly regular graphs with new parameters. We also discuss some of the links between partially balanced designs, $2$-adesigns (which were recently coined by Cunsheng Ding in ``Codes from Difference Sets'' (2015)), and partial geometric designs, and make an investigation into when a $2$-adesign is partial geometric.

\smallskip

\end{abstract}

%\begin{keywords}
%Difference sets, almost difference sets, cyclotomic cosets, combinatorial designs.
%\end{keywords}

\section{Introduction}\label{sec1}
In this correspondence we discuss partial geometric difference sets and partial geometric difference families, which were introduced by Olmez in \cite{O} and Nowak et al. \cite{KN}. Here it was also shown that partial geometric difference sets and partial geometric difference families give partial geometric designs. It is clear that partial geometric designs have several applications in graph theory, coding theory and cryptography \cite{BOSE}, \cite{BOS}, \cite{OLMEZ}, \cite{O2}. It was shown by Brouwer et al. in \cite{BOS} that directed strongly regular graphs can be obtained from partial geometric designs. In \cite{O2} Olmez showed that certain partial geometric difference sets can be used to construct plateaued functions. Recently, a generalization of combinatorial designs related to almost difference sets, namely the {\it t-adesign}, was introduced and their applications to constructing linear codes discussed \cite{CUN}, \cite{ADF}, \cite{MI}. In this correspondence we construct several new families of partial geometric difference sets and partial geometric difference families thereby giving new directed strongly regular graphs, and we also discuss some of their links to partially balanced designs and $2$-adesigns, and make an investigation into when an $2$-adesign is partial geometric.
\par
This correspondence is organized as follows. Section \ref{sec2} recalls several preliminary concepts that will be used throughout the sequel. In Section \ref{sec3} we construct four families of partial geometric difference sets, in Section \ref{sec4} we construct six families of partial geometric difference families, in Section \ref{sec5} we discuss some of the links between partially balanced designs, $2$-adesigns, and partial geometric designs, and investigate when a $2$-adesign is partial geometric. Section \ref{sec6} concludes the paper.

\section{Preliminaries}\label{sec2}

\subsection{Partial Geometric Designs}

 An {\it incidence structure} is a pair $(V,\mathcal{B})$ where $V$ is a finite set of points and $\mathcal{B}$ is a finite set blocks composed of points of $V$. For a given point $u\in V$, its {\it replication number} is the number of blocks of $\mathcal{B}$ in which it occurs, and is denoted by $r_{u}$. Given two distinct points $u,w\in V$, their {\it index} is the number of blocks in which they occur together, and is denoted $r_{uw}$. A {\it tactical configuration} is an incidence structure $(V,\mathcal{B})$ where the cardinalities of blocks in $\mathcal{B}$ and the replication numbers of points in $V$ are both constant.
 \par
 Let $(V,\mathcal{B})$ be a tactical configuration where $|V|=v$, each block has cardinality $k$, and each point has replication number $r$. For each point $u\in V$ and each block $b\in \mathcal{B}$, let $s(u,b)$ denote the number of flags $(w,c)\in V\times \mathcal{B}$ such that $w\in b\setminus\{u\},u\in c$ and $c\neq b$. If there are integers $\alpha'$ and $\beta'$ such that \[
 s(u,b)=\begin{cases} \alpha', \text{ if } u\in b, \\
                      \beta', \text{ otherwise},\end{cases} \] as $(u,b)$ runs over $V\times\mathcal{B}$, then we say that $(V,\mathcal{B})$ is a {\it partial geometric design} with parameters $(v,k,r;\alpha',\beta')$.

 \subsection{Partial Geometric Difference Sets and Partial Geometric Difference Families}

Let $G$ be a finite (additive) Abelian group and $S\subset G$. We denote the multiset $\left[x-y \mid x,y\in S\right]$ by $\Delta (S)$. For a family $\mathcal{S}=\{S_{1},...,S_{n}\}$ of subsets of $G$, we denote the multiset union $\bigsqcup_{i=1}^{n}\Delta (S_{i})$ by $\Delta (\mathcal{S})$. For a subset $S\subset G$ we denote $|\{(x,y)\in S\times S \mid z=x-y\}|$ by $\delta_{S}(z)$. For a family $\mathcal{S}=\{S_{1},...,S_{n}\}$ of subsets of $G$ we denote $|\{(x,y)\in S_{i}\times S_{i} \mid x = x-y\}|$ by $\delta_{i}(z)$.
 \par
 Let $v,k$ and $n$ be integers with $v>k>2$. Let $G$ be a group of order $v$. Let $\mathcal{S}=\{S_{1},...,S_{n}\}$ be a collection of distinct $k$-subsets of $G$. If there are constants $\alpha,\beta$ such that for each $x\in G$ and each $i\in \{1,...,n\}$, \[
 \sum_{i=1}^{n}\sum_{y\in S_{i}}\delta_{S_{i}}(x-y)=\begin{cases} \alpha, \text{ if } x\in S_{i}, \\
                                                         \beta, \text{ otherwise}, \end{cases}\] then we say $\mathcal{S}$ is a {\it partial
 geometric difference family} with parameters $(v,k,n;\alpha,\beta)$. When $n=1$ and $\mathcal{S}=\{S\}$, we simply say that $S$ is a {\it partial geometric difference set} with parameters are $(v,k;\alpha,\beta)$.
 \par
 Again let $G$ be a group and $\mathcal{S}=\{S_{1},...,S_{n}\}$ a collection of distinct $k$-subsets of $G$. We call the set of translates $\bigcup_{i=1}^{n}\{S_{i}+g\mid g\in G\}$ the {\it development} of $\mathcal{S}$, and denote it by $Dev(\mathcal{S})$. The following theorem was proved in \cite{KN}.
 \begin{theorem} Let $\mathcal{S}=\{S_{1},...,S_{n}\}$ be a collection of distinct $k$-subsets of a group $G$ of order $v$. If $\mathcal{S}$ is a partial geometric difference family with parameters $(v,k,n;\alpha,\beta)$, then $(G,Dev(\mathcal{S}))$ is a partial geometric design with parameters $(v,k,kn;\alpha',\beta')$ where $\alpha'=\sum_{i=1}^{n}\sum_{y\in S_{i}\setminus\{x\}}(\delta_{S_{i}}(x-y)-1)$ for $x\notin \mathcal{S}$, and $\beta'=\sum_{i=1}^{n}\sum_{y\in S_{i}}\delta_{S_{i}}(x-y)$ for $x\in \mathcal{S}$ (see Remark \ref{re2.0}).
 \end{theorem}
 \begin{remark}\label{re2.0} The parameters for the corresponding partial geometric designs seem to disagree in Lemma $2.4$ of \cite{O} and Theorem $3$ of \cite{KN}. To see this, the reader should compare Definition $2.2$ of \cite{O} to Definition $1$ of \cite{KN}.
 \end{remark}

 \subsection{Strongly Regular Graphs and Digraphs}\label{ssec2.0}
 In this correspondence all graphs are assumed to be loopless and simple. Let $\Gamma$ be an undirected graph with $v$ vertices. Let $A$ denote the adjacency matrix of $\Gamma$. Then $\Gamma$ is called {\it strongly regular} with parameters $(v,k,\lambda,\mu)$ if \[
 A^{2}=kI+\lambda A + \mu (J-I-A) \text{ and } AJ=JA=kJ. \] A directed graph $\Gamma$ with adjacency matrix $A$ is said to be {\it directed strongly regular} with parameters $(v,k,t,\lambda,\mu)$ if \[
 A^{2}=tI+\lambda A + \mu (J-I-A) \text{ and } AJ=JA=kJ.\]
 The following theorems were proved in \cite{BOS}.
 \begin{theorem} Let $(V,\mathcal{B})$ be a tactical configuration, and let $\Gamma$ be the directed graphs with vertex set \[
 \mathcal{V}=\{(u,b)\in V\times \mathcal{B}\mid u\notin V\}\] and adjacency given by \[
 (u,b)\rightarrow (w,c) \text{ if and only if } u\in c.\] Then $\Gamma$ is directed strongly regular if and only if $(V,\mathcal{B})$ is partial geometric.
 \end{theorem}
 \begin{theorem}Let $(V,\mathcal{B})$ be a tactical configuration, and let $\Gamma$ be the directed graphs with vertex set \[
 \mathcal{V}=\{(u,b)\in V\times \mathcal{B}\mid u\in V\}\] and adjacency given by \[
 (u,b)\rightarrow (w,c) \text{ if and only if } (u,b)\neq (w,c) \text{ and } u\in c.\] Then $\Gamma$ is directed strongly regular if and only if $(V,\mathcal{B})$ is partial geometric.
 \end{theorem}

 \subsection{Group Ring Notation}
 For any finite group $G$ the {\it group ring} $\mathbb{Z}\left[G\right]$ is defined as the set of all formal sums of elements of $G$, with coefficients in $\mathbb{Z}$. The operations ``$+$'' and ``$\cdot$'' on $\mathbb{Z}\left[G\right]$ are given by \[
\sum_{g\in G}a_{g}g+\sum_{g\in G}b_{g}g=\sum_{g\in G}(a_{g}+b_{g})g \] and \[
\left(\sum_{g\in G}a_{g}g\right)\left(\sum_{h\in G}b_{h}h\right)=\sum_{g,h\in G}a_{g}b_{h}(g+h).\] where are $a_{g},b_{g}\in \mathbb{Z}$.
\par
The group ring $\mathbb{Z}\left[G\right]$ is a ring with multiplicative identity $\mathbf{1}=\underline{Id}$, and for any subset $X\subset G$, we denote by $\underline{X}$ the sum $\sum_{x\in X}x$, and we denote by $\underline{X}^{-1}$ the sum $\sum_{x\in X}(-x)$.

\subsection{Cyclotomic Classes and Cyclotomic Numbers}
Let $q$ be a prime power, and $\gamma\in \mathbb{F}_{q^{2}}$ primitive. The {\it cyclotomic classes} of order $e$ are given by $C_{i}^{e}=\gamma^{i}\langle \gamma^{e} \rangle$ for $i=0,1,...,e-1$. Define $(i,j)=|C_{i}^{e}\cap (C_{j}^{e}+1)|$. It is easy to see there are at most $e^{2}$ different cyclotomic numbers of order $e$. When it is clear from the context, we simply denote $(i,j)_{e}$ by $(i,j)$. We will need the following lemma.

\begin{lemma}\label{le2.0} {\rm \cite{NOW}} Let $q=em+1$ be a prime power for some positive integers $e$ and $f$. In the group ring $\mathbb{Z}\left[\mathbb{F}_{q}\right]$ we have \[
\underline{C_{i}^{e}}\underline{C_{j}^{e}}=a_{ij}\mathbf{1}+\sum_{k=0}^{e-1}(j-i,k-i)_{e}\underline{C_{k}^{e}}\] where \[
a_{ij}=\begin{cases} f, \text{ if } m \text{ is even and } j=i, \\
                     f, \text{ if } m \text{ is odd and } j=i+\frac{e}{2}, \\
                     0, \text{ otherwise. }\end{cases}\]
\end{lemma}

\section{New Partial Geometric Difference Sets}\label{sec3}

We will need the following lemmas.

\begin{lemma}\label{le3.1} {\rm \cite{BAUM}} Let $q$ be a prime power and let $C_{i}$, for $i=0,1,...,q$ denote the cyclotomic classes of order $q+1$ in $\mathbb{F}_{q^{2}}$. Then the cyclotomic numbers are given by \begin{eqnarray*}
(0,0) & = & q-2, \\
(i,i)=(i,0)=(0,i) & = & 0, \\
(i,j) & = & 1, (i\neq j). \end{eqnarray*}
\end{lemma}

\begin{lemma}\label{le3.0} {\rm \cite{KN}} Let $p$ be a prime and let $C_{i}$, for $i=0,1,...,p$ denote the cyclotomic classes of order $p+1$ in $\mathbb{F}_{p^{2}}$. Let $S_{i}=C_{i}\cup\{0\}$ for $i=0,1,...,p$. If $x\notin S_{j}$ then $|(x-S_{j})\cap C_{i}|=1$ for each $i\in\{0,1,...,p\}\setminus\{j\}$.
\end{lemma}
The following is our first construction.

\begin{theorem}\label{th3.0} Let $p$ be a prime and let $C_{i}$, for $i=0,1,...,p$ denote the cyclotomic classes of order $p+1$ in $\mathbb{F}_{p^{2}}$. Let $S_{i}=C_{i}\cup\{0\}$ for $i=0,1,...,p$. Let $i',j'\in \{0,1,...,p\}$ be fixed with $i'\neq j'$. Let $m\equiv 0 ($mod $2)$ be a positive integer, and define \[
\Sigma_{l}=l+\{0,2,...,m-2\} \subset \mathbb{Z}_{m} \text{ for } l=0,1.\] Then $S_{i'j'}=\Sigma_{0}\times S_{i'}\cup\Sigma_{1}\times S_{j'}$ is a partial geometric difference set in $(\mathbb{Z}_{m}\times \mathbb{F}_{p^{2}},+)$ with parameters $(mp^{2},mp;(\frac{m}{2})^{2}p(p+3),\frac{3}{4}m^{2}p)$.
\end{theorem}
\begin{proof}
First note that $S_{0}\cong\mathbb{F}_{p}$, and $S_{i'},S_{j'}$ are both subgroups of $(\mathbb{F}_{p^{2}},+)$. For each $z\in \mathbb{F}_{p^{2}}$ and each $i\in\{0,1,..,p\}$ we have \begin{eqnarray}\label{eq3.2}
\delta_{S_{i}}&=&\begin{cases} |S_{i}|, \text{ if } z\in S_{i}, \\
                             0, \text{ otherwise}, \end{cases} = \begin{cases} p, \text{ if } z\in S_{i}, \\
                                                                               0, \text{ otherwise}.\end{cases}\end{eqnarray}
Now suppose that $(h,z)\in S_{i'j'}$. Then we have \begin{eqnarray}\label{eq3.0}
(h,z)-S_{i'j'} & = & \begin{cases} \Sigma_{0}\times S_{i'}\cup \Sigma_{1}\times (z-S_{j'}), \text{ if } h\in \Sigma_{0},z\in S_{i'}, \\
                                   \Sigma_{1}\times (z-S_{i'})\cup \Sigma_{0}\times S_{j'}, \text{ if } h\in \Sigma_{1},z\in S_{j'}.\end{cases}\end{eqnarray}
Denote the number of occurrences of $u$ in $\Delta(S_{i'j'})$ by $n_{u}$. Then $\sum_{(h',z')\in S_{i'j'}}\delta_{S_{i'j'}}((h,z)-(h',z'))$ can be written \[ \begin{cases}\sum_{v\in \sum_{0}\times\{0\}}n_{v}+\sum_{v\in \Sigma_{0}\times (S_{i'}\setminus\{0\})}n_{v}+\sum_{v\in \Sigma_{1}\times((z-S_{j'})\setminus\{z\})}n_{v}+\sum_{v\in\Sigma_{1}\times\{z\}}n_{v},\text{ if } h\in\Sigma_{0},z\in S_{i'},\\
                        \sum_{v\in\Sigma_{0}\times\{0\}}n_{v}+\sum_{v\in\Sigma_{0}\times(S_{j'}\setminus\{0\})}n_{v}
                        +\sum_{v\in\Sigma_{1}\times((z-S_{i'})\setminus\{z\})}n_{v}+\sum_{v\in\Sigma_{1}\times\{z\}}n_{v}, \text{ if }h\in \Sigma_{1},z\in S_{j'}.\end{cases} \]
which, by (\ref{eq3.0}), in both cases gives $\alpha=\frac{m}{2}p+(\frac{m}{2})^{2}(p-1)p+(\frac{m}{2})^{2}p+(\frac{m}{2})^{2}p=(\frac{m}{2})^{2}p(p+3)$.
\par
Now suppose that $(h,z)\notin S_{i'j'}$. We have \begin{eqnarray}\label{eq3.1}
(h,z)-S_{i'j'} & = & \begin{cases} \Sigma_{0}\times(z-S_{i'})\cup\Sigma_{1}\times S_{j'}, \text{ if } h\in \Sigma_{0},z\in S_{j'},\\
                                   \Sigma_{1}\times S_{i'}\cup \Sigma_{0}\times (z-S_{j'}), \text{ if } h\in \Sigma_{1},z\in S_{i'},\\
                                   \Sigma_{0}\times(z-S_{i'})\cup \Sigma_{1}\times(z-S_{j'}),\text{ if } h\in \Sigma_{0},z\notin S_{i'}\cup S_{j'},\\
                                   \Sigma_{1}\times(z-S_{i'})\cup\Sigma_{0}\times(z-S_{j'})\text{ if } h\in\Sigma_{1},z\notin S_{i'}\cup
S_{j'}.\end{cases}\end{eqnarray} Using Lemma \ref{le3.0}, it is easy to see that $(h,w)$, where $h\in \Sigma_{0}$, $w\in (z-S_{i'})\cap S_{j'}$ and $z\notin S_{i'}$, appears $\frac{m}{2}p$ times in $\Delta(S_{i'j'})$, and each member of $\Sigma_{1}\times S_{j'}$ appears $m$ times. Similarly, $(0,w)$, where $w\in S_{i'}\cap(z-S_{j'})$ and $z\notin S_{j'}$, appears $(\frac{m}{2})^{2}p$ times in $\Delta(S_{i'j'})$, and each member of $\Sigma_{1}\times S_{i'}$ appears $m$ times.
\par
We need to consider the two cases $h\in \Sigma_{0},z\notin S_{i'}\cup S_{j'}$ and $h\in \Sigma_{1},z\notin S_{i'}\cup S_{j'}$. Notice \begin{eqnarray*} \underline{\Delta(S_{i'j'})}=\underline{S_{i'j'}S_{i'j'}^{-1}} & = & (\underline{\Sigma_{0}\times S_{i'}\cup \Sigma_{1}\times S_{j'}})(\underline{\Sigma_{0}\times S_{i'}\cup \Sigma_{1}\times S_{j'}})^{-1} \\
& = & \frac{m}{2}(\Sigma_{0},\underline{S_{i'}S_{i'}^{-1}})+\frac{m}{2}(\Sigma_{1},\underline{S_{i'}S_{j'}^{-1}})
+\frac{m}{2}(\Sigma_{1},\underline{S_{j'}S_{i'}^{-1}})+\frac{m}{2}(\Sigma_{0},\underline{S_{j'}S_{j'}^{-1}}) \\
& = & \frac{m}{2}(\Sigma_{0},p\underline{(S_{i'}\cup S_{j'})})+(\Sigma_{1},\underline{S_{i'}S_{j'}^{-1}}+\underline{S_{j'}S_{i'}^{-1}})\end{eqnarray*}

Since $C_{i}\cap C_{j}=\emptyset$ for $i\neq j$ and $f=\frac{p^{2}-1}{p+1}$ is even, we have by Lemma \ref{le2.0} that \begin{equation*}
\underline{C_{i}C_{j}^{-1}}=\underline{C_{i}C_{j}}=\sum_{l=0}^{p}(j-i,l-i)\underline{C_{l}}=\sum_{l\neq i,j}(j-i,l-i)\underline{C_{l}}. \end{equation*} Thus, by using (\ref{eq3.1}) and Lemma \ref{le3.1}, we can see that in the case where $h\in \Sigma_{0}$ and $z\notin S_{i'}\cup S_{j'}$, each member of $\Sigma_{1}\times (z-S_{j'})$ appears $m$ times in $\Delta(S_{i'j'})$, and a member $(h,w)$, where $h\in \Sigma_{0}$ and $w\in (z-S_{i'})\cap S_{j'}$, appears $\frac{m}{2}p$ times. Similarly, in the case where $h\in \Sigma_{1}$ and $w\notin S_{i'}\cup S_{j'}$, each member of $\Sigma_{1}\times (z-S_{i'})$ appears $m$ times in $\Delta(S_{i'j'})$, and a member $(h,w)$, where $h\in \Sigma_{0}$ and $w\in S_{i'}\cap(z-S_{j'})$, appears $\frac{m}{2}p$ times. Thus we have $\beta=\sum_{(h',z')\in S_{i'j'}}\delta_{S_{i'j'}}((h,z)-(h',z'))=(\frac{m}{2})^{2}p+m(\frac{m}{2})p=\frac{3}{4}m^{2}p$.
\end{proof}
We give another construction of partial geometric difference sets in products of Abelian groups.

\begin{theorem}\label{th3.1} Let $p$ be a prime and let $C_{i}$, for $i=0,1,...,p$, denote the cyclotomic classes of order $p+1$ in $\mathbb{F}_{p^{2}}$. Let $S_{i}=C_{i}\cup\{0\}$ for $i=0,1,...,p$. Let $i',j'\in \{0,1,...,p\}$ be fixed with $i'\neq j'$. Let $\{0,3\},\{1,4\}\subset \mathbb{Z}_{6}$. Then $S_{i'j'}=\{0,3\}\times S_{i'}\cup\{1,4\}\times S_{j'}$ is a partial geometric difference set in $(\mathbb{Z}_{6}\times \mathbb{F}_{p^{2}},+)$ with parameters $(6p^{2},4p;20p,8p)$.
\end{theorem}
\begin{proof} We have already established in (\ref{eq3.2}) that \[ \delta_{S_{i}}(z)=\begin{cases} p, \text{ if } z\in S_{i},\\
                                                                                                   0, \text{ otherwise}.\end{cases}\] Now suppose
that $(h,z)\in S_{i'j'}$. Then we have \begin{eqnarray}\label{eq3.3}
(h,z)-S_{i'j'} & = & \begin{cases} \{0,3\}\times S_{i'}\cup \{2,5\}\times(z-S_{j'}), \text{ if } h\in \{0,3\},z\in S_{i'}, \\
                                   \{1,4\}\times(z-S_{i'})\cup \{0,3\}\times S_{j'}, \text{ if } h\in \{1,4\},z\in S_{j'}.\end{cases}\end{eqnarray}
If we denote the number of occurrences of $u$ in $\Delta(S_{i'j'})$ by $n_{u}$ then, using (\ref{eq3.3}), we have \begin{eqnarray*} \alpha=\sum_{(h',z')\in S_{i'j'}}\delta_{S_{i'j'}}((h,z)-(h',z'))& = & \sum_{v\in\{0,3\}\times\{0\}}n_{v}+\sum_{v\in\{0,3\}\times(S_{i'}\setminus\{0\})}n_{v}+\sum_{v\in\{2,5\}\times(z-S_{j'})}n_{v}\\
& = & 8p+8p+4p \\
& = & 20p.\end{eqnarray*} There are seven expressions for $(h,z)-S_{i'j'}$ depending on whether $h$ is contained in $\{0,3\},\{1,4\}$ or $\{2,5\}$, and whether $z$ is contained in $S_{i'}$ or $S_{j'}$, or contained in neither. These are simple to compute and we do not list them. Using Lemma \ref{le3.0} it is easy to see that $(h,w)$, where $h\in \{0,3\}$ and $w\in (z-S_{j'})$ for $z\notin S_{i'}$, appears $2p$ times in $\Delta(S_{i'j'})$, and each member of $\{2,5\}\times S_{j'}$ appears $2p$ times. The cases where $h\in \{1,4\},z\in S_{i'}$, where $h\in\{2,5\},z\in S_{j'}$, and where $h\in\{2,5\},z\in S_{i'}$, are similar to the previous case. We need to consider the cases where $z\notin S_{i'}\cup S_{j'}$. Notice \begin{eqnarray*} \underline{\Delta(S_{i'j'})}=\underline{S_{i'j'}S_{i'j'}^{-1}}
& = & (\underline{\{0,3\}\times S_{i'}\cup \{1,4\}\times S_{j'}})(\underline{\{0,3\}\times S_{i'}\cup \{1,4\}\times S_{j'}})^{-1} \\
& = & 2(\{0,3\},\underline{S_{i'}S_{i'}^{-1}})+2(\{2,5\},\underline{S_{i'}S_{j'}^{-1}})+2(\{1,4\},\underline{S_{j'}S_{i'}^{-1}})+
2(\{0,3\},\underline{S_{j'}S_{j'}^{-1}}) \\
& = & 2((\{0,3\},p\underline{(S_{i'}\cup S_{j'})})+(\{2,5\},\underline{S_{i'}S_{j'}^{-1}})+(\{1,4\},\underline{S_{j'}S_{i'}^{-1}}))\end{eqnarray*}

Since $C_{i}\cap C_{j}=\emptyset$ for $i\neq j$ and $f=\frac{p^{2}-1}{p+1}$ is even, we have by Lemma \ref{le2.0} that \begin{equation*}
\underline{C_{i}C_{j}^{-1}}=\underline{C_{i}C_{j}}=\sum_{l=0}^{p}(j-i,l-i)\underline{C_{l}}=\sum_{l\neq i,j}(j-i,l-i)\underline{C_{l}}. \end{equation*} Thus, by using the expressions for $(h,z)-S_{i'j'}$ and Lemma \ref{le3.1}, we can see that, in the case where $h\in\{0,3\}$ and $z\notin S_{i'}\cup S_{j'}$, each member of $\{2,5\}\times(z-S_{j'})$ appears twice in $\Delta(S_{i'j'})$, and a member $(h,w)$, where $h\in\{0,3\}$ and $w\in(z-S_{i'})\cap S_{j'}$, appears $2p$ times. The cases where $h\in\{1,4\},z\notin S_{i'}\cup S_{j'}$ and where $h\in\{2,5\},z\notin S_{i'}\cup S_{j'}$ are similar. Thus we can conclude that $\beta=\sum_{(h',z')\in S_{i'j'}}\delta_{S_{i'j'}}((h,z)-(h',z'))=8p$.
\end{proof}
We next construct partial geometric designs from planar functions. For a more detailed introduction to planar functions the reader is referred to \cite{ACH} and \cite{CUN}.

Let $(A,+)$ and $(B,+)$ be Abelian groups of order $n$ and $m$ respectively. Let $f:A\rightarrow B$ be a function. One measure of the nonlinearity of $f$ is given by $P_{f}=\max\limits_{0\neq a\in A}\max\limits_{b\in B}Pr(f(x+a)-f(x)=b)$, where $Pr(E)$ denotes the probability of the event $E$. The function $f$ is said to have {\it perfect nonlinearity} if $P_{f}=\frac{1}{m}$.
The following lemma gives many examples of perfect nonlinear functions in finite fields.
\begin{lemma} {\rm \cite{ACH}} The power function $x^{s}$ from $\mathbb{F}_{p^{m}}$ to $\mathbb{F}_{p^{m}}$, where $p$ is an odd prime, has perfect nonlinearity $P_{f}=\frac{1}{p^{m}}$ for the following values of $s$: \begin{enumerate}
\item $s=2$,
\item $s=p^{k}+1$, where $m/gcd(m,k)$ is odd,
\item $s=(3^{k}+1)/2$, where $p=3$, $k$ is odd, and $gcd(m,k)=1$.\end{enumerate}
\end{lemma}
We will use the following lemma.
\begin{lemma}\label{le3.2} {\rm \cite{ACH}} Let $f$ be a function from an Abelian group $(A,+)$ of order $n$ to another Abelian group $(B,+)$ of order $n$ with perfect nonlinearity $P_{f}=\frac{1}{n}$. Define $C_{b}=\{x\in A\mid f(x)=b\}$ and $C=\bigcup_{b\in B}\{b\}\times C_{b}\subset B\times A$. Then \[|C\cap(C+(w_{1},w_{2}))|=\begin{cases} n, \text{ if } (w_{1},w_{2})=(0,0), \\
                                         0, \text{ if } w_{1}\neq 0, w_{2}=0, \\
                                         1, \text{ otherwise}.\end{cases}\]
\end{lemma}
The following is a construction.
\begin{theorem}\label{th3.2} Let $f$ be a function from an Abelian group $(A,+)$ of order $n$ to another Abelian group $(B,+)$ of order $n$ with perfect nonlinearity $P_{f}=\frac{1}{n}$. Define $C_{b}=\{x\in A\mid f(x)=b\}$ and $C=\bigcup_{b\in B}\{b\}\times C_{b}\subset B\times A$. Then $C$ is a partial geometric difference set in $A\times B$ with parameters $(n^{2},n;2n-1,n-1)$.
\end{theorem}
\begin{proof} Suppose $(h,z)\in C$. Then we have \begin{eqnarray*} (h,z)-C & = & \bigcup_{b\in B}\{h-b\}\times (z-C_{b}) \\
                                                                           & = & \bigcup_{b\in B}\{h-b\}\times \{z-x\mid z,x\in A,f(x)=b\}. \end{eqnarray*}
Denote the number of occurrences of $u$ in $\Delta(C)$ by $n_{u}$. Define $V_{1}=\bigcup_{b\in B\setminus\{h\}}\{h-b\}\times\{0\}$ and
$V_{2}=\bigcup_{b\in B}\{h-b\}\times\{z-x\mid x,z\in A,z\neq x,f(x)=b\}$. Then, using Lemma \ref{le3.2}, we have \begin{eqnarray*}
\sum_{(h',z')\in C}\delta((h,z)-(h',z')) & = & n_{(0,0)}+\sum_{v\in V_{1}}n_{v}+\sum_{v\in V_{2}}n_{v} \\
                                         & = & n+0+(n-1)\\
                                         & = & 2n-1.\end{eqnarray*}
Now suppose $(h,z)\notin C$. Define $U=\bigcup_{b\in B\setminus\{f(0)\}}\{h-b\}\times \{z-x\mid x\in A,f(x)=b\}$. Then, using Lemma \ref{le3.2}, we have \begin{eqnarray*} \sum_{(h',z')\in C}\delta((h,z)-(h',z')) & = & n_{(h-f(0),0)}+\sum_{v\in U}n_{v} \\
                                                           & = & 0 + (n-1) \\
                                                           & = & n-1.\end{eqnarray*}
\end{proof}
\begin{corollary} Let $f(x)=x^{s}$ be a function from $\mathbb{F}_{p^{m}}$ to $\mathbb{F}_{p^{m}}$, where $p$ is an odd prime. Define $C_{b}=\{x\in \mathbb{F}_{p^{m}}\mid f(x)=b\}$ and $C=\bigcup_{b\in \mathbb{F}_{p^{m}}}\{b\}\times C_{b}\subset \mathbb{F}_{p^{m}}\times \mathbb{F}_{p^{m}}$. If: \begin{enumerate}
\item $s=2$,
\item $s=p^{k}+1$, where $m/gcd(m,k)$ is odd, or
\item $s=(3^{k}+1)/2$, where $p=3$, $k$ is odd, and $gcd(m,k)=1$.\end{enumerate} Then $(\mathbb{F}_{p^{m}}\times \mathbb{F}_{p^{m}},Dev(C))$ is a partial geometric difference set with parameters $(p^{2m},p^{m};2p^{m}-1,p^{m}-1)$.
\end{corollary}
\begin{remark} Interestingly, the partial geometric difference sets constructed in Theorem \ref{th3.2} are almost difference sets {\rm \cite{ACH}} and so correspond to planar $2$-adesigns (see Section \ref{sec5}).
\end{remark}
Nowak et al., in \cite{KN}, constructed partial geometric difference families in groups $G=\mathbb{Z}_{n}$ where $n=4l$ for some positive integer $l$. We close this section by further generalizing this idea. The proof is a simple counting exercise, and so is omitted.
\begin{theorem}\label{th3.3} Let $G=\mathbb{Z}_{2}\times \mathbb{Z}_{n}$ where $n=4l$ for some positive integer $l$. Let $H=\langle 4 \rangle$ be the unique subgroup of $\mathbb{Z}_{n}$ of order $l$. Define $H+i=\{z+i\mid z\in H\}=\{x\in \mathbb{Z}_{n}\mid x\equiv i ($mod $4)\}$ for $i=0,1,2,3$ (i.e. the cosets of $H$ in $\mathbb{Z}_{n}$).  Then both $\{0\}\times (H\cup(H+1))\cup\{1\}\times(H\cup(H+3))$ and $\{1\}\times (H\cup(H+1))\cup\{0\}\times(H\cup(H+3))$ are partial geometric difference sets in $G$ with parameters $(8l,4l;6l^{2},10l^{2})$.
\end{theorem}

We next discuss some new partial geometric difference families.

\section{New Partial Geometric Difference Families}\label{sec4}

We begin with the following construction.
\begin{theorem}\label{th4.1} Let $n=p^{u}$ where $p$ is an odd prime and $u\geq 2$ is an integer. Let $S=\{0,1,...,p^{u-1}-1\}$ and, for $l=0,1,...,p^{u-1}$, define $S_{l}=(pl-1)S=\{0,pl-1,2pl-2,...,(p^{u-1}-1)pl-(p^{u-1}-1)\}$. Then $S=\{S_{l}\mid l=1,2,...,p^{u-1}\}$ is a partial geometric difference family with parameters $(p^{u},p^{u-1},p^{u-1};(p^{u-1}-1)p^{u-1}p^{u-2},p^{u}+(p^{u-1}-1)p^{u-1}p^{u-2})$.
\end{theorem}
\begin{proof} We will use the following property, which is easily seen to hold: \begin{multline}\label{eq4.1} \text{{\it The members} } \pm(pl-1),\pm(2pl-2),...,\pm(p^{u-1}-1)pl-(p^{u-1}-1) \text{ {\it each appear}}\\ \text{ {\it in the multiset }} \Delta(S_{l}) \text{ {\it with multiplicities} } p^{u-1}-1,p^{u-1}-2,...,1 \text{ {\it respectively}.\quad\quad\quad\quad\quad\quad\quad\quad}\end{multline} Also note that for $s\in S$ we have that $\pm s(pl-1)\equiv\mp(p^{u-1}-s)(pl-1)($mod $p^{u-1})$ for each $l$, $l=0,1,...,p^{u-1}$.
\\
{\bf Claim:} For each $s\in S$, the equation $s(pl-1)\equiv\alpha($mod $p^{u})$ has $p^{u-2}$ solutions $(s,l)$ for $s,l\in\{0,1,...,p^{u-1}\}$.
\\
{\bf Proof of Claim:} Notice if $pl-1\equiv\alpha($mod $p^{u})$ we have $s(p(l+\beta)-1)\equiv\alpha($mod $p^{u})$ if and only if \begin{equation}\label{eq4.0} s\alpha+sp\beta\equiv\alpha({\rm \text{mod }} p^{u}).\end{equation}  We can see that (\ref{eq4.0}) holds if and only if $s\equiv1($mod $p)$. We know that $(s,\beta)=(1,0)$ is a solution. Now set $s=1+p$ and $\beta=pl'+1$ for some $l'\in \{0,...,p^{u-1}\}$. Then we have \begin{eqnarray*}
                          (p+1)(pl-1)+(p+1)(pl'+1)p=pl-1& \Leftrightarrow & p\alpha+(p^{2}l'+p+pl'+1)p=0\\
                          & \Leftrightarrow & p(pl-1)+p^{2}(1+l')+p=0\\
                          & \Leftrightarrow & p^{2}(l+l'+1)=0\\
                          & \Leftrightarrow & l+l'+1\equiv p({\rm \text{mod }} p^{u-1}).\end{eqnarray*} Thus we can choose $\beta=pl'+1$ where
$l'\equiv-l-1($mod $p^{u-1})$ and we have a solution. Since there are $p^{u-1}$ such solutions, the claim is proved.
\\
Thus we have that each element of $\mathbb{Z}_{p^{u}}$ not congruent to $0($mod $p^{u-1})$ appears in $S_{l}$ for $p^{u-2}$ different values of $l$. Since $0$ appears $p^{u}$ times in the multiset $\bigsqcup_{l=1}^{p^{u}-1}\Delta(S_{l})$, and by (\ref{eq4.1}), we must have that if $x\in S_{l}$ for some $l$ then \begin{align*}\alpha=\sum_{l}\sum_{y\in S_{l}}\delta_{S_{l}}(x-y) & =  p^{u}+\sum_{l}\sum_{y\in S_{l},x\neq y}p^{u-2}p^{u-1}&\\
                                                                    & =  p^{u}+(p^{u-1}-1)p^{u-2}p^{u-1}, &(\text{since each } S_{l}\text{ contains }0) \end{align*}
Now notice that if we reduce the elements of $S_{l}$ modulo $p^{u-1}$ we get the set $S=S_{0}$. It follows then that for any $S_{l}$, and any $x\notin S_{l}$, the set $x-S_{l}$ contains exactly one member congruent to $0($mod $p^{u-1})$. Then we have that if $x\notin S_{l}$ for all $l$, then
\begin{align*}\beta=\sum_{l}\sum_{y\in S_{l}}\delta_{S_{l}}(x-y) & =  \sum_{l}\sum_{y\in S_{l},x\not\equiv y({\rm\text{mod }}p^{u-1})}p^{u-2}p^{u-1}&\\
                                                         & =  (p^{u-1}-1)p^{u-2}p^{u-1}. & \end{align*}
\end{proof}
Our next two constructions further generalize Theorem \ref{th3.0}.
\begin{theorem}\label{th4.0} Let $p$ be a prime, and for each $i\in \{0,1,...,p\}$ let $S_{i}=C_{i}\cup\{0\}$ where $C_{i}$ is the $i$th cyclotomic class of order $p+1$ in $\mathbb{F}_{p^{2}}$. Let $I\subset\{0,1,...,p\}$ such that $|I|=2\kappa$ for some positive integer $\kappa$. Say $I=\{i_{1},...,i_{2\kappa}\}$, and define $\Theta_{0}$ to be the set of all pairs $(i,j)\in I\times I$ such that $i\neq j$ and each member of $I$ appears in exactly one ordered pair. Let $m$ be a positive, even integer. Define $\Sigma_{l}=l+\{0,2,...,m-2\}\subset \mathbb{Z}_{m}$ for $l=0,1$, and for each $(i',j')\in \Theta_{0}$ define $S_{i'j'}=\Sigma_{0}\times S_{i'}\cup\Sigma_{1}\times S_{j'}\subset \mathbb{Z}_{m}\times\mathbb{F}_{p^{2}}$. Then $S=\{S_{i'j'}\mid (i',j')\in \Theta_{0}\}$ is a partial geometric difference family with parameters $(mp^{2},mp,\kappa;(\frac{m}{2})^{2}p(p+3)+(\kappa-1)\frac{3}{4}m^{2}p,\kappa\frac{3}{4}m^{2}p)$.
\end{theorem}
\begin{proof} We have already established in Theorem \ref{th3.0} that if $(h,z)\in S_{i'j'}$ then \[
\sum_{(h',z')\in S_{i'j'}}\delta_{S_{i'j'}}((h,z)-(h',z'))=(\frac{m}{2})^{2}p(p+3),\] and if $(h,z)\notin S_{i'j'}$ then \[
\sum_{(h',z')\in S_{i'j'}}\delta_{S_{i'j'}}((h,z)-(h',z'))=\frac{3}{4}m^{2}p.\] Then if $(h,z)\in S_{i'j'}$ for some $(i',j')\in\Theta_{0}$ we have \[
\alpha=\sum_{(i,j)\in\Theta_{0}}\sum_{(h',z')\in S_{ij}}\delta_{S_{ij}}((h,z)-(h',z'))=(\frac{m}{2})^{2}p(p+3)+(\kappa-1)\frac{3}{4}m^{2}p,\] and if $(h,z)\notin S_{i'j'}$ for all $(i'j')\in\Theta_{0}$ we have \[
\beta=\sum_{(i,j)\in\Theta_{0}}\sum_{(h',z')\in S_{ij}}\delta_{S_{ij}}((h,z)-(h',z'))=\kappa\frac{3}{4}m^{2}p.\]
\end{proof}
The proofs of the following corollaries are omitted as they use simple counting principals similar to those of Theorem \ref{th4.0}.
\begin{corollary}\label{co4.0} Let $p$ be a prime, and for each $i\in \{0,1,...,p\}$ let $S_{i}=C_{i}\cup\{0\}$ where $C_{i}$ is the $i$th cyclotomic class of order $p+1$ in $\mathbb{F}_{p^{2}}$. Let $I\subset\{0,1,...,p\}$ such that $|I|=2\kappa$ for some positive integer $\kappa$. Say $I=\{i_{1},...,i_{2\kappa}\}$, and define $\Theta_{1}=\{(i_{2\kappa},i_{1}),(i_{1},i_{2}),(i_{2},i_{3}),...,(i_{2\kappa-1},i_{2\kappa})\}$. Let $m$ be a positive integer. Define $\Sigma_{l}=l+\{0,2,...,m-2\}\subset \mathbb{Z}_{m}$ for $l=0,1$, and for each $(i',j')\in \Theta_{1}$ define $S_{i'j'}=\Sigma_{0}\times S_{i'}\cup\Sigma-{1}\times S_{j'}\subset \mathbb{Z}_{m}\times\mathbb{F}_{p^{2}}$. Then $S=\{S_{i'j'}\mid (i',j')\in \Theta_{1}\}$ is a partial geometric difference family with parameters $(mp^{2},mp,\kappa;(\frac{m}{2})^{2}p(p+3)+(2\kappa-1)\frac{3}{4}m^{2}p,\kappa\frac{3}{2}m^{2}p)$.
\end{corollary}
\begin{corollary}\label{co4.1} Let $p$ be a prime, and for each $i\in \{0,1,...,p\}$ let $S_{i}=C_{i}\cup\{0\}$ where $C_{i}$ is the $i$th cyclotomic class of order $p+1$ in $\mathbb{F}_{p^{2}}$. Let the integer $\kappa$ and $\Theta_{e}$, for $e=0$ resp. $1$, be defined as in Theorem \ref{th4.0} resp. Corollary \ref{co4.0}. For each $(i',j')\in \Theta_{e}$ define $S_{i'j'}^{e}=\{0,3\}\times S_{i'}\cup\{1,4\}\times S_{j'}\subset \mathbb{Z}_{6}\times \mathbb{F}_{p^{2}}$ for $e=0,1$, and $\mathcal{S}^{e}=\{S_{i'j'}^{e}\mid(i',j')\in\Theta_{e}\}$. Then $\mathcal{S}^{0}$ resp. $\mathcal{S}^{1}$ is a partial geometric difference family with parameters $(6p^{2},4p,\kappa;20p+8(\kappa-1)p,8\kappa p)$ resp. $(6p^{2},4p,2\kappa;20p+8(2\kappa-1)p,16\kappa p)$.
\end{corollary}
We close this section with the following construction.
\begin{theorem}\label{th4.2} Let $G$ be an Abelian group of odd composite order $n$. Let $H$ be a proper, nontrivial subgroup of $G$ of order $m$, and set $\kappa=\frac{n/m-1}{2}$. Suppose that $g_{1},...,g_{\kappa}\in G$ are such that $\{H\pm g_{i}\mid 1\leq i\leq\kappa\}$ is a partition of $G\setminus H$. Then $\mathcal{S}=\{H\cup(H+g_{i})\mid 1\leq i\leq \kappa\}$ is a partial geometric difference family with parameters $(n,2m,\kappa;n+m(m-1),2m^{2})$.
\end{theorem}
\begin{proof} If $h\in H\cup(H+g_{i'})$ for some fixed $i'\in \{1,...,\kappa\}$ then \begin{equation}\label{eq4.2}
g-(H\cup(H+g_{i'}))=(H+g)\cup(H+(g-g_{i'}))=\begin{cases} (H+g_{i'})\cup H, \text{ if } h\in H+g_{i'},\\
                                                        H\cup(H-g_{i'}) \text{ otherwise}.\end{cases}\end{equation}
If $g\notin H\cup(H+g_{i})$ for all $i\in\{1,...,\kappa\}$ then, for each $i$, \begin{equation}\label{eq4.3}
 g-(H\cup(H+g_{i}))=(H+g)\cup(H+(g-g_{i}))\end{equation} where $H+g$ and $H+(g-g_{i})$ are distinct members of $\{H\pm g_{i}\mid 1\leq i\leq\kappa\}$. Also notice that, for fixed $i'\in\{1,...,\kappa\}$, we have \begin{eqnarray}\label{eq4.4}
 \Delta(H\cup(H+g_{i'})) & = & \underline{(H\cup(H+g_{i'}))(H\cup(H+g_{i'}))}^{-1}\nonumber \\
                         & = & 2\underline{HH}^{-1}+\underline{H(H+g_{i'})}^{-1}+\underline{(H+g_{i'})H}^{-1} \nonumber \\
                         & = & 2m\underline{H}+m\underline{(H-g_{i'})}+m\underline{(H+g_{i'})}.\end{eqnarray}
 Let $n_{u}$ denote the number of occurrences of $u$ in $\Delta(\mathcal{S})=\bigsqcup_{i=1}^{\kappa}\Delta(H\cup(H+g_{i}))$. Then, using (\ref{eq4.2}) and (\ref{eq4.4}), if $g\in \mathcal{S}$ we have \begin{eqnarray*}
 \alpha=\sum_{i=1}^{\kappa}\sum_{g'\in H\cup(H+g_{i})}\delta_{H\cup(H+g_{i})}(g-g') & = & 2\kappa|H|+m|H\pm g+{i}|\\
                                                                                    & = & (\frac{n}{m}-1)m+m^{2} \\
                                                                                    & = & n+m(m-1),\end{eqnarray*} and,
 using (\ref{eq4.3}) and (\ref{eq4.4}), if $g\notin \mathcal{S}$ we have \begin{eqnarray*}
 \beta=\sum_{i=1}^{\kappa}\sum_{g'\in H\cup(H+g_{i})}\delta_{H\cup(H+g_{i})}(g-g') & = & 2m^{2}.\end{eqnarray*}
\end{proof}
\begin{center}
\captionof{table}{{\scriptsize Parameters of partial geometric difference sets constructed in this paper.}}\label{ta4.0}
\begin{tabular}{|c|c|c|c|}
  \hline
   Reference & $(v,k;\alpha,\beta)$& Group &Information \\
  \hline\hline
 Theorem \ref{th3.0}&$(mp^{2},mp;(\frac{m}{2})^{2}p(p+3),\frac{3}{4}m^{2}p)$&$\mathbb{Z}_{m}\times\mathbb{F}_{p^{2}}$&$m$ even, $p$ an odd prime \\\hline
 Theorem \ref{th3.1}&$(6p^{2},4p;20p,8p)$&$\mathbb{Z}_{6}\times\mathbb{F}_{p^{2}}$&$p$ an odd prime  \\\hline
 Theorem \ref{th3.2}&$(n^{2},n;2n-1,n-1)^{*}$&\shortstack{$A\times B$\\(generic)}&$A,B$ both Abelian groups of order $n$ \\\hline
 Theorem \ref{th3.3}&$(8l,4l;6l^{2},10l^{2})$&$\mathbb{Z}_{2}\times\mathbb{Z}_{n}$&$n=4l$ for positive integer $l$\\\hline
  \multicolumn{4}{l}{{\scriptsize *: This partial geometric difference set is an almost difference set and corresponds to a planar $2$-adesign (see Section \ref{sec5}).}}
\end{tabular}
\end{center}
\newpage
\begin{center}
\captionof{table}{{\scriptsize Parameters of partial geometric difference families constructed in this paper.}}\label{ta4.1}
\begin{tabular}{|c|c|c|c|}
  \hline
   Reference & $(v,k;\alpha,\beta)$& Group &Information \\
  \hline\hline
  Theorem \ref{th4.1}&\shortstack{$(p^{u},p^{u-1},p^{u-1};\alpha,\beta)$\\$\alpha=p^{u}+(p^{u-1}-1)p^{u-1}p^{u-2}$\\$\beta=(P^{u-1}-1)p^{u-1}p^{u-2}$}&$\mathbb{Z}_{p^{u}}$&\shortstack{$p$ an odd prime,\\$u\geq 2$ an integer} \\\hline
 Theorem \ref{th4.0}&\shortstack{$(mp^{2},mp,\kappa;(\frac{m}{2})^{2}p(p+3)+(\kappa-1)\frac{3}{4}m^{2}p,\kappa\frac{3}{4}m^{2}p)$\\$1\leq\kappa\leq\frac{p+1}{2}$}&$\mathbb{Z}_{m}\times\mathbb{F}_{p^{2}}$&
 \shortstack{$m$ even, \\$p$ an odd prime} \\\hline
 Corollary \ref{co4.0}&\shortstack{$(mp^{2},mp,\kappa;(\frac{m}{2})^{2}p(p+3)+(2\kappa-1)\frac{3}{4}m^{2}p,\kappa\frac{3}{2}m^{2}p)$\\$1\leq\kappa\leq\frac{p+1}{2}$}&$\mathbb{Z}_{m}\times\mathbb{F}_{p^{2}}$&
 \shortstack{$m$ even, \\$p$ an odd prime} \\\hline
 Corollary \ref{co4.1}&\shortstack{$(6p^{2},4p,\kappa;20p+8(\kappa-1)p,8\kappa p)$\\$1\leq\kappa\leq\frac{p+1}{2}$}&$\mathbb{Z}_{6}\times\mathbb{F}_{p^{2}}$&$p$ an odd prime \\\hline
 Corollary \ref{co4.1}&\shortstack{$(6p^{2},4p,2\kappa;20p+8(2\kappa-1)p,16\kappa p)$\\$1\leq\kappa\leq\frac{p+1}{2}$}&$\mathbb{Z}_{6}\times\mathbb{F}_{p^{2}}$&$p$ an odd prime \\\hline
 Theorem \ref{th4.2}&\shortstack{$(n,2m,\kappa;n+m(m-1),2m^{2})$\\$\kappa=\frac{n/m-1}{2}$}&\shortstack{$G$\\(generic)}&\shortstack{$G$ Abelian of odd,\\composite order\\$n$ with $m|n$}\\\hline
\end{tabular}
\end{center}
\section{Partial Geometric Designs, Partially Balanced Designs, Adesigns, and Their Links}\label{sec5}
We first discuss an important connection between partial geometric designs and tactical configurations that have exactly {\it two} indices, i.e., tactical configurations $(V,\mathcal{B})$ where there are integers $\mu_{1}\neq\mu_{2}$ such that for any pair of distinct points $x,y\in V$, $r_{xy}\in\{\mu_{1},\mu_{2}\}$. If $A$ is the $v\times b$ incidence matrix of a tactical configuration $(V,\mathcal{B})$ with $v$ points, $b$ blocks, and the two indices $\mu_{1}\neq\mu_{2}$, then we will denote by $A_{1}$ the symmetric matrix whose $(i,j)$th entry is $1$ if the points corresponding to the $i$th and $j$th rows of $A$ are contained in exactly $\mu_{1}$ blocks, and is $0$ otherwise. We will need the following lemma.

\begin{lemma}\label{le6.0} {\rm \cite{NEUM}} An incidence structure $(V,\mathcal{B})$ is a partial geometric design with parameters $(v,k,r;\alpha',\beta')$ if and only if its incidence matrix $A$ satisfies \[
AJ=rJ,JA=kJ \text{ and } AA^{T}A=n'A+\alpha'J,\] where $n'=r+k+\beta'-\alpha'-1$.
\end{lemma}
Suppose $(V,\mathcal{B})$ is a partial geometric design with parameters $(v,k,r;\alpha',\beta')$ and the two indices $\mu_{1}\neq \mu_{2}$. Let $A$ be the incidence matrix of $(V,\mathcal{B})$. It is easy to see that $A$ satisfies \begin{equation}\label{eq6.0}
AA^{T}=(r-\mu_{2})I+(\mu_{1}-\mu_{2})A_{1}+\mu_{2}(J-A_{1}-I).\end{equation} Since $(V,\mathcal{B})$ is partial geometric, by Lemma \ref{le6.0} we have that $A$ also satisfies \begin{equation}\label{eq6.1}
n'A+\alpha'J=AA^{T}A=(r-\mu_{2})A+(\mu_{1}-\mu_{2})A_{1}A+\mu_{2}kJ.\end{equation}
Then, using (\ref{eq6.0}) and (\ref{eq6.1}), we must have that $A_{1}A=\nu A+\zeta (J-A)$ for some integers $\nu$ and $\zeta$. Moreover we must have $n'+\alpha'=r-\mu_{2}+\nu$ and $\alpha'=\zeta+\mu_{2}k$. This means that, for each pair $(x,b)\in V\times \mathcal{B}$, we have \begin{equation}\label{eq6.2}
|\{y\in b\mid y\neq x,r_{xy}=\mu_{1}\}|=\begin{cases} \nu\quad (=n'+\alpha'-r+\mu_{2}), \text{ if } x\in b,\\
                                                      \zeta\quad (=\alpha'-\mu_{2}k), \text{ otherwise}.\end{cases}\end{equation}
Note that condition (\ref{eq6.2}) is necessary and sufficient.
\par
 Now set $\sigma=r-\mu_{2}$, $\phi=\mu_{1}-\mu_{2}$ and $\psi=\nu-\zeta$. Then we can write $AA^{T}=\sigma I+\phi A_{1}+\mu_{2} J$ and $A_{1}A=\psi A+\zeta J$. By Lemma \ref{le6.0}, and since $A_{1}$ is symmetric, we have \[
krJ=AA^{T}J=\psi A_{1}J+\sigma J+\mu_{2}kJ=JAA^{T}.\]
Then, after some simple arithmetic, we can get \begin{equation}\label{eq6.3} A_{1}J=JA_{1}=\kappa J \end{equation} where $\kappa=\frac{(k-1)r+\mu_{2}(1-v)}{\mu_{1}-\mu_{2}}$. Now set $\epsilon = \zeta r-\mu_{2}(\kappa-\psi)$. Then we have \begin{eqnarray}\label{eq6.4}
(\psi A +\zeta J)A^{T}=A_{1}AA^{T}=\phi A_{1}^{2}+\sigma A_{1}+\mu_{2}\kappa J & \Leftrightarrow & \phi A_{1}^{2}+\sigma A_{1}-\psi\phi A_{1}+\psi\sigma I=\epsilon J \nonumber\\
                                                                               & \Leftrightarrow & A_{1}^{2}
=k'I+aA_{1}+b(J-I-A_{1}),\end{eqnarray} where $k'=\kappa=\frac{\epsilon-\psi\sigma}{\phi},a=\frac{\epsilon+\psi\phi-\sigma}{\phi}$ and $b=\frac{\epsilon}{\phi}$ are integers (note that $k'=\kappa$ follows from (\ref{eq6.3})). From (\ref{eq6.3}) and (\ref{eq6.4}) it is clear that $A_{1}$ is the adjacency matrix of a strongly regular graph with parameters $(v,k',a,b)$ (see Subsection \ref{ssec2.0}). We have thus shown the following.
\begin{lemma}\label{th6.0} A tactical configuration with the two indices $\mu_{1}\neq\mu_{2}$ and incidence matrix $A$ is partial geometric with parameters $(v,k,r;\alpha',\beta')$ if and only if there are integers $\nu$ and $\zeta$ such that for each pair $(x,b)\in V\times\mathcal{B}$,\[  |\{y\in b\mid y\neq x,r_{xy}=\mu_{1}\}|=\begin{cases} \nu\quad (=n'+\alpha'-r+\mu_{2}), \text{ if } x\in b,\\
                                                      \zeta\quad (=\alpha'-\mu_{2}k), \text{ otherwise},\end{cases}\] and $A_{1}$ is the
adjacency matrix of a strongly regular graph with parameters $(v,k',a,b)$ where $k'=\frac{\epsilon-\psi\sigma}{\phi},a=\frac{\epsilon+\psi\phi-\sigma}{\phi}$ and $b=\frac{\epsilon}{\phi}$. Moreover, $k'=\frac{(k-1)r+\mu_{2}(1-v)}{\mu_{1}-\mu_{2}}$.
\end{lemma}
It is interesting that condition (\ref{eq6.2}), when combined with (\ref{eq6.0}), leads to the strongly regular graph described by (\ref{eq6.4}). We can see that Lemma \ref{th6.0} describes a special class of block designs that have two indices. We now discuss a particular subclass of these block designs.
\par
A tactical configuration $(V,\mathcal{B})$ with the two indices $\mu_{1}$ and $\mu_{2}$ such that $\mu_{1}-\mu_{2}=1$ is called a $2$-{\it adesign}. Adesigns were were recently introduced in \cite{CUN}, and reported on in \cite{ADF} and \cite{MI}, where several constructions are given, and codes generated by the incidence matrices are computed.
\begin{theorem}\label{co6.0} A $2$-$(v,k,\lambda)$ adesign with incidence matrix $A$ is partial geometric with parameters $(v,k,r;\alpha',\beta')$ if and only if there are integers $\nu$ and $\zeta$ such that for each pair $(x,b)\in V\times\mathcal{B}$,\[  |\{y\in b\mid y\neq x,r_{xy}=\mu_{1}\}|=\begin{cases} \nu\quad (=n'+\alpha'-r+\lambda), \text{ if } x\in b,\\
                                                      \zeta\quad (=\alpha'-\lambda k), \text{ otherwise},\end{cases}\] and $A_{1}$ is the
adjacency matrix of a strongly regular graph with parameters $(v,k',a,b)$ where $k'=\epsilon-\psi\sigma,a=\epsilon+\psi-\sigma$ and $b=\epsilon$. Moreover, the following relations hold: $\sigma=r-\lambda,\psi=n'-r+\lambda(k+1),\epsilon=\lambda^{2}v+\lambda(k+r-2kr+\beta'-\alpha'-1)+\alpha'r$ and $k'=(k-1)r+\lambda(1-v)$.
\end{theorem}
We can see that Theorem \ref{co6.0} describes a special class of $2$-adesigns. There seem to be even fewer examples of these, and the few examples we can find have long since been discovered.
\begin{example} Let $(V,\mathcal{B})$ be a quasi-symmetric design with intersection numbers $s_{1}$ and $s_{2}$ such that $s_{2}-s_{1}=1$. Several families of such quasi-symmetric designs are known to exist {\rm \cite{PAW}}. It is well-known that the dual of any balanced incomplete block design is a partial geometric design {\rm \cite{OLMEZ}}. Then the dual $(V,\mathcal{B})^{\perp}$ of $(V,\mathcal{B})$ is a partial geometric $2$-adesign.
\end{example}
\begin{example} Let $p$ be an odd prime. Let $D_{i}^{p+1}$ denote the $i$th cyclotomic class or order $p+1$ in $\mathbb{F}_{p^{2}}$. It was shown in {\rm \cite{KN}} that $(\mathbb{F}_{p^{2}},Dev(D_{i}^{p+1}))$ is a partial geometric design. It is easy to see that $(\mathbb{F}_{p^{2}},Dev(D_{i}^{p+1}))$ has the two indices $\mu_{1}=1$ and $\mu_{2}=0$ (see {\rm \cite{NOW}}). Then $(\mathbb{F}_{p^{2}},Dev(D_{i}^{p+1}))$ is a symmetric partial geometric $2$-adesign.
\end{example}
\begin{example} Let $C$ be the partial geometric difference set from Theorem \ref{th3.2} in the Abelian group $A\times B$ of order $n^{2}$. Then $(A\times B,Dev(C))$ is a partial geometric design, and it was shown in {\rm \cite{ACH}} that $(A\times B,Dev(C))$ has the two indices $\mu_{1}=1$ and $\mu_{2}=0$. Then $(A\times B,Dev(C))$ is a symmetric partial geometric $2$-adesign.
\end{example}
\section{Concluding Remarks}\label{sec6}
We have constructed several families of partial geometric difference sets and partial geometric difference families which are recorded in Table \ref{ta4.0} and Table \ref{ta4.1} respectively. These families have new parameters and so give directed strongly regular graphs with new parameters. We discussed some links between partially balanced designs, $2$-adesigns, and partial geometric designs and made an investigation into when a $2$-adesign is partial geometric. The condition noted in Lemma \ref{th6.0} seems surprisingly strong, and describes a special class of partial geometric designs that correspond (via (\ref{eq6.4})) to strongly regular graphs. The condition noted in Theorem \ref{co6.0} is also strong and describes a special class of $2$-adesigns. We wonder whether Lemma \ref{th6.0} gives an indirect but viable way of searching for strongly regular graphs with new parameters.
\bibliographystyle{plain}
\bibliography{myref4}

\end{document}